\documentclass[preprint,11pt]{elsarticle}
\usepackage{mathrsfs}
\usepackage{amsfonts}
\usepackage{}
\usepackage{amsmath,amssymb,amsthm}
\usepackage{graphicx}
\usepackage{color}
\usepackage{soul} % 导入 soul 包
\usepackage{color, xcolor} % 颜色包，color 必须导入
\usepackage{bbm}
\usepackage{amssymb}
\usepackage{indentfirst}
\usepackage{latexsym,amsfonts,amssymb,amsmath}
\usepackage[curve,arrow,matrix,tips]{xy}
\UseRawInputEncoding

\makeatletter
\def\ps@pprintTitle{%
  \let\@oddhead\@empty
  \let\@evenhead\@empty
  \def\@oddfoot{\reset@font\hfil\thepage\hfil}
  \let\@evenfoot\@oddfoot}
\makeatother

\newtheorem{defin}{Definition}[section]

\newtheorem{theorem}[defin]{Theorem}
\newtheorem{lemma}[defin]{Lemma}
\newtheorem{proposition}[defin]{Proposition}
\newtheorem{corollary}[defin]{Corollary}
\newtheorem{qu}{Question}

\newtheorem{remark}[]{Remark}
\begin{document}

\begin{frontmatter}

\title{The quotient spaces of topological groups with a $q$-point\corref{cor2}}
%\tnotetext[mytitlenote]{This work was supported by}

%\author{Shou Lin}
%\ead{shoulin60@163.com}

\author[mymainaddress]{Li-Hong Xie}%\fnref{fn1}
\ead{yunli198282@126.com}

\author[mymainaddress]{Hai-Hua Lin}
\ead{1048363420@qq.com}

\author[mymainaddress1]{Piyu Li\corref{cor1}}
\ead{lpy91132006@aliyun.com}

%\author[mymainaddress]{Peng-Fei Yan} %\fnref{fn1}
%\ead{ypengfei@sina.com}

\cortext[cor1]{Corresponding author}
\cortext[cor2]{The project is supported by the Natural Science Foundation of Guangdong Province under Grant (No. 2021A1515010381) and the Innovation Project of Department of Education of Guangdong Province (No. 2022KTSCX145).}
%\address{Institute of Mathematics, Ningde Normal University, Ningde 352100,
%P. R. China}
\address[mymainaddress]{School of Mathematics and Computational Science, Wuyi University, Jiangmen, Guangdong, 529000, China}
\address[mymainaddress1]{School of Mathematics and Statistics, Xuzhou Institute of Technology, Xuzhou 221018, China}

\begin{abstract}
In this paper, we study the uniformities on the double coset spaces in topological groups. As an implication, the quotient spaces of topological groups with a $q$-point are studied. It mainly shows that: (1) Suppose that $G$ is a topological group  with a $q$-point and $H$ is a closed subgroup of $G$; then the quotient space $G/H$ is an open and quasi-perfect preimage of a metrizable space; in particular, $G/H$ is an $M$-space. (2) Suppose that $G$ is a topological group with a strict $q$-point and $H$ is a closed subgroup of $G$; then the quotient space $G/H$ is an open and sequentially perfect preimage of a metrizable space. (3) Suppose that $G$ is a topological group with a strong $q$-point and $H$ is a closed subgroup of $G$; then the quotient space $G/H$ is an open and strongly sequentially perfect preimage of a metrizable space. 
\end{abstract}

\begin{keyword}
Topological groups; double coset spaces; $q$-points; strict $q$-points; strong $q$-points

\MSC[2020] 54A20; 54H11; 54B15; 54C10; 54E15
\end{keyword}

\end{frontmatter}

\section{Introduction}
A group $G$ endowed with a topology $\tau$ is called a {\it paratopological group} if multiplication in $G$ is continuous as a mapping of
$G\times G$ to $G$, where $G\times G$ carries the usual product topology. A {\it topological group} is a paratopological group with continuous inversion. All topological groups in this paper are assumed to be Hausdorff.

A Tychnoff space $X$ is \v{C}ech-complete if $X$ is homeomorphic to a $G_\delta$-set
in a compact space. Feathered topological groups are very important in topological algebra. Recall that a topological group G is feathered if $G$ has a non-empty compact set $K$ of countable character in $G$. All \v{C}ech-complete topological groups ˇ
are feathered. It is also clear that all metrizable topological groups are feathered. It
is a deep result of M.M. Choban that a feathered topological group is \v{C}ech-complete if and only if it is Ra\v{\i}kov complete. B.A. Pasynkov proved that a topological group $G$ is feathered if and only if $G$ contains a compact subgroup $H$ such that the quotient space $G/H$ is metrizable. By this result, we have that a topological group $G$ is feathered iff $G$ is a paracompact $p$-space. Further more, Arhangel'skii and Tkachenko proved that: suppose that $H$ is a closed subgroup of a feathered group $G$. Then the quotient space $G/H$ is a paracompact space \cite[Theorem 4.3.23]{AT1}. In fact, in view of the proof of this theorem, they have proved the following result:

\begin{theorem} \label{the1.1}
Suppose that $H$ is a closed subgroup of a feathered group $G$. Then the quotient space $G/H$ is a perfect preimage of a metrizable space.
\end{theorem}

Recall that a continuous mapping $f:X\rightarrow Y$ is ({\it quasi}-) {\it perfect} if $f$ is closed and $f^{-1}(y)$ is (countably compact) compact in $X$ for each $y\in Y$. According to E. Michael \cite{Mii}, a point $x\in X$ is called a {\it $q$-point} of a space $X$ if there exists a sequence $\{U_n:n\in\omega\}$ of open neighborhoods of $x$ in $X$ such that every sequence $\{x_n\}_{n\in\omega}$ of points in $X$ such that $x_n\in U_n$ for each $n\in\omega$ has a point of accumulation in $X$. We also call $\{U_n:n\in\omega\}$ a {\it $q$-sequence} at $x$. A space $X$ is said to be a {\it$q$-space} if every point of it is a $q$-point. Recently, Peng and Liu obtained the following result:

\begin{theorem}\cite[Theorem 6]{PY}\label{PY}
 A topological group $G$ is a $q$-space if and only if $G$ is an open quasi-perfect preimage of a metrizable space.
\end{theorem}

In the same paper, they also introduced the class of strong $q$-spaces and gave a characterization of strong $q$-spaces in topological groups in terms of preimages of metrizable spaces under mappings satisfying some conditions (see \cite[Theorem 24]{PY}). A point $x\in X$ is called a {\it strong $q$-point} of a space $X$ if there exists a sequence $\{U_n:n\in\omega\}$ of open neighborhoods
of $x$ in $X$ such that every sequence $\{x_n\}_{n\in\omega}$ of points in $X$ such that $x_n\in U_n$ for each $n\in\omega$ has a convergent subsequence.
We also call $\{U_n:n\in\omega\}$ a strong {\it $q$-sequence} at $x$. A space $X$ is said to be a {\it strong $q$-space} if every point in $X$ is a strong $q$-point \cite{PY}.

In view of the definition of strong $q$-spaces, Lin, Xie and Chen \cite{LX} introduced the concept of strict $q$-spaces as following: A point $x\in X$ is called a {\it strict $q$-point} of a space $X$ if there exists a $q$-sequence $\{U_n:n\in\omega\}$ at $x$ such that $\bigcap_{n\in\omega} U_n$ is a sequentially compact set. We also call $\{U_n:n\in\omega\}$ a {\it strict $q$-sequence} at $x$. A space $X$ is said to be a {\it strict $q$-space} if every point in $X$ is a strict $q$-point.  Lin, Xie and Chen established the following results:

\begin{theorem}\cite[Theorem 2.6]{LX}
A topological group $G$ is a strong $q$-space if and only if $G$ is an open and strongly sequential-perfect preimage of a metrizable space.
\end{theorem}

A space $X$ is called {\it sequentially compact} if every sequence of $X$ has a convergent subsequence \cite{En}.
 A continuous mapping $f:X\rightarrow Y$ is called {\it strongly sequential-perfect} \cite{LX} if $f$ is closed and $f^{-1}(F)$ is a sequentially compact set for each sequentially compact set $F\subseteq Y$. A continuous mapping $f:X\rightarrow Y$ is called {\it sequential-perfect} \cite{LX} if $f$ is closed and $f^{-1}(y)$ is a sequentially compact set for each $y\in Y.$

\begin{theorem}\cite[Theorem 2.4]{LX}
A topological group $G$ is a strict $q$-space if and only if $G$ is an open sequential-perfect preimage of a metrizable space.
\end{theorem}

Hence, in view of Theorems \ref{the1.1}, it is natural to ask whether the quotient spaces of topological groups, which are $q$-spaces, strong $q$-spaces or strict $q$-spaces are certain preimages of metrizable spaces. These are the main aims of our paper.

%%%%%%%%%%%%%%%%%%%%%%%%%%%%%%%%%%%%%%%%%%%%%
%%%%%%%%%%%%%%%%%%%%%%%%%%%%%%%%%%%%%%%%%%%%%
\section{Preliminaries}\label{Sec:1}

We begin by introducing some notations. Let $X$ be a set and let $A$ and $B$ be subsets of $X\times X$, i.e. relations on the set $X$. The inverse relation of $A$ will be denoted by $-A$, i.e., $-A=\{(x,y):(y,x)\in A\}$. The composition of $A$ and $B$ will be denoted by $A+B$, i.e., $A+B=\{(x,z): \text{ there is a } y\in X \text{ such that } (x,y)\in A \text{ and } (y,z)\in B\}.$ $2A$ means that $A+A$.

The diagonal of the Cartesian product $X\times X$ is the set $\Delta_X=\{(x,x):x\in X\}$. Every set $V\subseteq X\times X$ that contains $\Delta_X$ and satisfies the condition $V=-V$ is called an {\it entourage of diagonal}. The family of all entourages of diagonal will be denoted by $\mathcal {D}_X$. Let $x_0$ be a point in $X$ and let $V\in \mathcal {D}_X$. The set $V[x_0]=\{x\in X:(x_0,x)\in V\}$ is called the {\it ball with center $x_0$} and {\it radius} $V$.

A {\it uniformity} on a set $X$ is a subfamily $\mathcal {U}\subseteq \mathcal {D}_X$ which satisfies the following conditions:
\begin{enumerate}
\item[$(U1)$] If $V\in \mathcal {U}$ and $V\subseteq W\in \mathcal {D}_X$, then $W\in \mathcal {U}$.
\item[$(U2)$] If $V_1,V_2\in \mathcal {U}$, then $V_1\cap V_2 \in \mathcal {U}$.
\item[$(U3)$] For every $V\in \mathcal {U} $, there exists a $W\in \mathcal {U}$ such that $2W\subseteq V$.
\end{enumerate}

A uniformity $\mathcal {U}$ on a set $X$ is called a {\it separated uniformity} if $\mathcal {U}$ satisfies the condition: $(U4)$ $\bigcap \mathcal {U}=\Delta_X.$

A subfamily $\mathcal {B}\subseteq \mathcal {U}$ is called a base for the uniformity $\mathcal {U}$ if for every $V\in \mathcal {U} $ there exists a $W\in\mathcal {B}$ such that $W\subseteq V$.

The following result is well known.
\begin{proposition}\label{p1}
Let $X$ be a set. For a subfamily $\mathcal {B}\subseteq \mathcal {D}_X$, if $\mathcal {B}$ satisfies the following conditions

\begin{enumerate}
\item[$(BU1)$] if $V_1,V_2\in \mathcal {B}$, there is $V\in\mathcal {B} $ such that $V\subseteq V_1\cap V_2 $;
\item[$(BU2)$] for every $V\in \mathcal {B} $, there exists a $W\in \mathcal {B}$ such that $2W\subseteq V$;
\end{enumerate}
then $\mathcal {B}$ is a base for a uniformity $\mathcal {U}$ on the set $X$.
 If the family $\mathcal {B}$ satisfies the additional condition $(BU3)$ $\bigcap \mathcal {B}=\Delta_X$, then $\mathcal {B}$ is a base for a separated uniformity $\mathcal {U}$ on the set $X$.
\end{proposition}

The following result is well known.
\begin{theorem}\cite[Theorem 8.1.1]{En}\label{t1}
For every uniformity $\mathcal {U}$ on a set $X$ the family $\mathcal {O}=\{U\subseteq X:\text{ for every }x\in U \text{ there is a }V\in\mathcal {U} \text{ such that }V[x]\subseteq U\}$ is a topology on the set.
\end{theorem}

The topology $\mathcal {O}$ is called the {\it topology induced by the uniformity $\mathcal {U}$}. A topological space $X$ with a topology $\mathcal {T}$ is {\it (pseudo)metrizable} if the topology  $\mathcal {T}$ is induced by a (pseudo)metric $d$ on the set $X$.

The following result is well known.
\begin{theorem}\label{T1}
A topological space $X$ with a topology $\mathcal {T}$ is (pseudometrizable) metrizable if and only if the topology $\mathcal {T}$ is induced by a (uniformity) separated uniformity with a countable base.
\end{theorem}

Let $H$ and $K$ be subgroups
of an abstract group $G$. Put $$K\backslash G/H = \{KxH : x\in G\}
$$ of subsets of $G$. It is clear that for any $x, y \in G$, either $KxH = KyH$ or $KxH \cap KyH = \emptyset$.
Therefore, $K\backslash G/H$ is a partition of the group $G$. Denote by $\pi$ the canonical mapping of
$G$ onto $K\backslash G/H$ defined by $\pi(x) = KxH$, for every $x \in G$. If $G$ is a topological group, then we say {\it double coset space} is the $K\backslash G/H$ with the quotient topology of $G$.

let $\mathscr{N}_s(e)$ be the family of all open symmetric neighborhoods of the identity
$e$ in a topological group $G$. Denote by $\pi$ the natural mapping of $G$ onto the quotient space $ Z=K\backslash G/H$. For every $V\in  \mathscr{N}_s(e)$, put
$$E_V^r=\{(\pi(x),\pi(y)): \pi(y)\in \pi(Vx)\},$$
$$E_V^l=\{(\pi(x),\pi(y)): \pi(y)\in \pi(xV)\},$$
$$\mathscr{B}_Z^r=\{E_V^r: V\in \mathscr{N}_s(e)\}$$
and
$$\mathscr{B}_Z^l=\{E_V^l: V\in \mathscr{N}_s(e)\}.$$

%%%%%%%%%%%%%%%%%%%%%%%%%%%%%%%%%%%%%%%%%%%%%
%%%%%%%%%%%%%%%%%%%%%%%%%%%%%%%%%%%%%%%%%%%%%
\section{The uniformities on the double coset spaces in topological groups}\label{Sec:2}
A subgroup $H$ of a topological group $G$ is called {\it neutral} if for every open
neighborhood $U$ of the identity $e$ in $G$, there exists an open neighborhood $V$
of $e$ such that $HV \subseteq UH$ (equivalently, $V H \subseteq HU$) \cite{Pon}.

A subgroup $H$ of a topological group $G$ is called {\it strongly neutral} if for arbitrary open neighborhood $O$ of $H$ in $G$ there is an open neighborhood $V$ of the identity such that $HV\subseteq O$ (equivalently, $V H \subseteq O$). It is obvious that every strongly neutral subgroup is neutral

\begin{proposition}\label{pro}
Let $K$ be a strongly neutral subgroup of topological group $G$. Then $KF$ and $FK$ are closed in $G$ for each closed set $F$ in $G$.
\end{proposition}

\begin{proof}
Take an arbitrary point $x\notin KF$. Then $G\setminus Fx^{-1}$ is an open neighborhood of $K$, so there is an open neighborhood $V$ of the identity such that $KV\subseteq G\setminus Fx^{-1}$. This implies that $KV\cap Fx^{-1}=\emptyset$ which is equivalent to $Vx\cap KF=\emptyset$. This implies that $KF$ is closed, because $Vx$ is an open neighborhood of $x$. Similarly, one can show that $FK$ is closed.
\end{proof}

According to Proposition \ref{pro} one can easily obtain the following result.

\begin{corollary}\label{cor}
Let $K$ be a strongly neutral subgroup of topological group $G$. Then natural quotient mappings $\pi:G\rightarrow G/K$ and $q:G\rightarrow K\backslash G$ are open and closed.
\end{corollary}

\begin{proposition}\label{pro1}
Let $H$ and $K$ be subgroups of topological group $G$. Then the natural mapping $\pi:G\rightarrow K\backslash G/H$ is open. If $K$ is strongly neutral and $H$ is closed, then the double coset spaces $K\backslash G/H$ and $H\backslash G/K$ are regular.
\end{proposition}

\begin{proof}
If $U$ is open in $G$, so is $KUH$ in $G$. Thus $\pi^{-1}(\pi(U))=KUH$ is open in $G$. $\pi(U)$ is open in $K\backslash G/H$, because $\pi$ is a quotient mapping. This implies that $\pi$ is open.

Take an arbitrary point $\pi(x)\in K\backslash G/H$, where $x\in G$. Then $\pi^{-1}(\pi(x))=KxH$ is closed in $G$ by Proposition \ref{pro}, because $K$ is strongly neutral and $xH$ is closed in $G$. $\pi(x)$ is closed in $G$ since $\pi$ is a quotient mapping.  This implies that $K\backslash G/H$ is a $T_1$-space.

Take any point $z=\pi(x)$ in $K\backslash G/H$ for some $x\in G$ and a neighbourhood $O$ of $z$. There exist open neighbourhoods $U$ and $V$ of the identity $e$ in $G$ such that $\pi(Ux)\subseteq O$ and $V^2\subseteq U$. Since $K$ is strongly neutral, one can choose an open symmetric neighbourhood $W$ of $e$ such that $WK\subseteq KV$. Clearly, $\pi(Wx)$ is an open neighbourhood of $\pi(x)$ in $K\backslash G/H$ such that $\pi(Wx)\subseteq O$, and we claim that the closure of $\pi(Wx)$ is contained in $O$.

Indeed, suppose that $\pi(Wy)\cap \pi(Wx)\neq\emptyset$ for some $y\in G$. Then $Wy\cap KWxH\neq\emptyset$. Thus we have: $$y\in WKWxH\subseteq KVWxH\subseteq KV^2xH\subseteq KUxH=\pi^{-1}\pi(Ux).$$
It follows that $\pi(y)\in \pi(Ux)\subseteq O $. Since $\pi(Wy)$ is an open neighbourhood of $\pi(y)$ in $K\backslash G/H$, we conclude that all accumulation points of $\pi(Wx)$ lie in $O$. Thus the space $K\backslash G/H$ is regular.

Similarly, one can prove that the space $H\backslash G/K$ is regular.
\end{proof}

\begin{theorem}\label{th1}
Let $K$ and $H$ be subgroups of a topological group $G$. If $K$ is neutral, then the family $\mathscr{B}_Z^r$ is a base for some right uniformity $\mathcal {U}_Z^r$ on the double coset space $Z = K\backslash G/H$, which induces the quotient topology on $Z$. If $H$ is
closed and $K$ is strongly neutral, then the family $\mathscr{B}_Z^r$ is a base for some right separated uniformity $\mathcal {U}_Z^r$ on the double coset space $Z = K\backslash G/H$, which induces the quotient topology on $Z$.
\end{theorem}

\begin{proof}
Let $K$ be neutral. First, to show that the family $\mathscr{B}_Z^r$ is a base for some right uniformity $\mathcal {U}_Z^r$, it is only to prove that the family $\mathscr{B}_Z^r$ satisfies the conditions $(BU1)$-$(BU2)$ in Proposition \ref{p1}. i.e.,
\begin{enumerate}
\item[$(BU1$)] for arbitrary $E_{V_1}^r, E_{V_2}^r\in \mathscr{B}_Z^r$, there is a $ E_{V}^r\in \mathscr{B}_Z^r$ such that $E_{V}^r\subseteq E_{V_1}^r\cap E_{V_2}^r$;
\item[$(BU2)$] for arbitrary $E_{V}^r\in \mathscr{B}_Z^r$, there is a $ E_{W}^r\in \mathscr{B}_Z^r$ such that $2E_{W}^r\subseteq E_{V}^r$.
%\item[$(U3)$] $\Delta=\bigcap \mathscr{B}_Z^r$, where $\Delta=\{(\pi(x),\pi(x)):x\in G\}$.
\end{enumerate}

For $(BU1)$, $E_{V_1\cap V_2}^r\subseteq E_{V_1}^r\cap E_{V_2}^r$ is evident, because $E_{U}^r\subseteq E_{W}^r$ whenever $U\subseteq W$.

For $(BU2)$, take arbitrary $E_{V}^r\in \mathscr{B}_Z^r$. For $V\in \mathscr{N}_s(e)$, choose a $O\in \mathscr{N}_s(e)$ such that $O^2\subseteq V$. Since $K$ is neutral, for $O$ there is a $W\in \mathscr{N}_s(e)$ such that $W\subseteq O$ and $WK\subseteq KO$. We claim that $2E_{W}^r\subseteq E_{V}^r$.

 Indeed, take arbitrary $(\pi(x),\pi(y))\in 2E_{W}^r$. Then there is $z\in G$ such that $(\pi(x),\pi(z))\in E_{W}^r$ and $(\pi(z),\pi(y))\in E_{W}^r$, which implies that $\pi(z)\in \pi(Wx)$ and $\pi(y)\in \pi(Wz)$, respectively. Thus, we have $z\in KWxH$ and $y\in KWzH$. This implies that $$y\in KWzH\subseteq  KWKWxHH\subseteq KKOWxHH=KOWxH\subseteq KO^2xH\subseteq KVxH.$$ Thus, $\pi(y)\in \pi(KVxH)=\pi(Vx)$, which implies that $(\pi(x),\pi(y))\in E_{V}^r$. Therefore, we have proved $2E_{W}^r\subseteq E_{V}^r$.
Thus, we have proved that the family $\mathscr{B}_Z^r$ is a base for some right uniformity $\mathcal {U}_Z^r$ on $Z$.

Next, we shall show that the topology $\mathcal {T}_1$ induced by the right uniformity $\mathcal {U}_Z^r$ on $Z$ is compatible with the quotient topology $\mathcal {T}_2$ on $Z$. On one hand, since $\pi$ is open by Proposition \ref{pro1}, $E_V^r[\pi(x)]=\pi(Vx)$ is open in $Z$ for each $x\in G$ and each $E_V^r\in \mathscr{B}_Z^r $. This implies that $\mathcal {T}_1\subseteq \mathcal {T}_2$, because the family $\mathscr{B}_Z^r$ is a base for the right uniformity $\mathcal {U}_Z^r$. On the other hand, take arbitrary non-empty open set $O$ in $\mathcal {T}_2$ and arbitrary point $\pi(x)\in O$ for some $x\in G$. Then $\pi^{-1}(O)$ is an open neighborhood of $x$ in $G$. One can find a $V\in \mathscr{N}_s(e)$ such that $Vx\subseteq \pi^{-1}(O)$. Thus $\pi(Vx)\subseteq O$ and $\pi(x)\in \pi(Vx)$. Since $\pi(Vx)=E_V^r[\pi(x)]$, we have $O$ is in $\mathcal {T}_1$. This shows that $\mathcal {T}_2\subseteq \mathcal {T}_1$.

Thus, we have proved that the topology induced by the right uniformity $\mathcal {U}_Z^r$ on $Z$ is compatible with the quotient topology on $Z$.

If $K$ is strongly neutral and $H$ is closed, then it is only to show that $\mathscr{B}_Z^r$ satisfies the conditions $(BU3)$ in Proposition \ref{p1}. i.e., $(BU3)$ $\Delta_Z=\bigcap \mathscr{B}_Z^r$.

For $(BU3)$, take arbitrary $\pi(x)\neq \pi(y)$ for some $x,y\in G$. Then $x\notin KyH$. Since $H$ is closed and $K$ is strongly neutral, $KyH$ is closed by Proposition \ref{pro}. Furthermore, there is a $U\in \mathscr{N}_s(e)$ such that $Ux\cap KyH=\emptyset$, which implies $x\notin UKyH$. For $U$, one can find a $V\in \mathscr{N}_s(e)$ such that $KV\subseteq UK$, because $K$ is strongly neutral and $UK$ is a neighborhood of $K$. We claim that $(\pi(y), \pi(x))\notin E_V^r$, which implies that $\Delta_Z=\bigcap \mathscr{B}_Z^r$. Indeed, if $(\pi(y), \pi(x))\in E_V^r$, then $\pi(x)\in \pi(Vy)$, which implies that $$x\in KVyH\subseteq UKyH.$$ This is a contradiction with $x\notin UKyH.$
\end{proof}

\begin{remark}\label{r}
From Theorem \ref{th1} and its proof it follows that:
\begin{enumerate}
\item[(1)] If $K=\{e\}$, then the family $\mathscr{B}_{G/H}^r$ is a base for some right uniformity $\mathcal {U}_{G/H}^r$ on the coset space $G/H$, which induces the quotient topology on $G/H$. In addition, if $H$ is closed, then the right uniformity $\mathcal {U}_{G/H}^r$ is separated.
\item[(2)] If $H=\{e\}$, then the family $\mathscr{B}_{K\backslash G}^r$ is a base for some right uniformity $\mathcal {U}_{K\backslash G}^r$ on the coset space $K\backslash G$, which induces the quotient topology on $K\backslash G$. In addition, if $K$ is closed, then the right uniformity $\mathcal {U}_{G/H}^r$ is separated.
\end{enumerate}
\end{remark}

Similarly, one can prove the following result.
\begin{theorem}\label{th11}
Let $K$ and $H$ be subgroups of a topological group $G$. If $H$ is neutral, then the family $\mathscr{B}_Z^l$ is a base for some left uniformity $\mathcal {U}_Z^l$ on the double coset space $Z = K\backslash G/H$, which induces the quotient topology on $Z$. If $K$ is
closed and $H$ is strongly neutral, then the family $\mathscr{B}_Z^l$ is a base for some left separated uniformity $\mathcal {U}_Z^l$ on the double coset space $Z = K\backslash G/H$, which induces the quotient topology on $Z$.
\end{theorem}

\begin{remark}\label{R}
From Theorem\ref{th11}, we can obtain that:
\begin{enumerate}
\item[(1)] If $K =\{e\}$, where $e$ is the identity in $G$, then the family $\mathscr{B}_{G/H}^l$ is a base for some left uniformity $\mathcal {U}_{G/H}^l$ on the quotient space $ G/H$. In addition, if $H$ is closed, then the left uniformity $\mathcal {U}_{G/H}^l$ is separated.
\item[(2)] If $H=\{e\}$, where $e$ is the identity in $G$, then the family $\mathscr{B}_{K\backslash G}^l$ is a base for some left uniformity $\mathcal {U}_{K\backslash G}^l$ on the quotient space $ K\backslash G$. In addition, if $K$ is closed, then the left uniformity $\mathcal {U}_{K\backslash G}^l$ is separated.
\end {enumerate}
\end{remark}

\begin{proposition}\label{Pro1}
Suppose that $K$ and $H$ are subgroups of a topological group $G$. If $K$ is strongly neutral, then the natural mapping $q_K:G/H\rightarrow K\backslash G/H$ defined by $q_K(\pi_H(x))=\pi(x)$, for each $x\in G$, is continuous, open and closed, where $\pi_H:G\longrightarrow G/H$ and $\pi:G\longrightarrow K\backslash G/H$ are the canonical mappings.
\end{proposition}

\begin{proof}
Put $Z=K\backslash G/H$ and let $\pi:G\rightarrow Z$ be the canonical mapping. Then $\pi=q_K\circ \pi_H$ and, since the mappings $\pi$ and $\pi_H$ are open and continuous, so is $q_K$.

 \[
\xymatrix{ G\ar@{>}[r]^{\pi_H} \ar@{>}[dr]^{\pi}
& G/H \ar@{>}[d]^{q_K}   \\
   & K\backslash G/H
}
\]
It remains to verify that $q_K$ is a closed mapping. Let $z=\pi(x)$ be a point of $Z$. Then $q_K^{-1}(z)=\pi_H(KxH)=\pi_H(Kx)$. Take an arbitrary open neighborhood $O$ of $\pi_H(Kx)$ in $G/H$. Then $Kx\subseteq \pi_H^{-1}(O)$. It is equivalent to $K\subseteq \pi_H^{-1}(O)x^{-1}$. Since $K$ is strongly neutral and $\pi_H^{-1}(O)x^{-1}$ is open, there is an open neighborhood $V$ at the identity in $G$ such that $KV\subseteq \pi_H^{-1}(O)x^{-1}$. This implies that $KVx\subseteq \pi_H^{-1}(O)$. In particular,  $\pi_H(KVx)\subseteq O$. It is obvious that $W=\pi(Vx)$ is an open neighborhood of $z$ in $Z$ which satisfies $$q_K^{-1}(W)=q_K^{-1}(\pi(Vx))=\pi_H(\pi^{-1}(\pi(Vx)))=\pi_H(KVxH)=\pi_H(KVx)\subseteq O.$$
We have proved that, for every neighborhood $O$ of the fiber $q_K^{-1}(z)$ in $G/H$, there exists an open neighborhood $W$ of $z$ in $Z$ satisfying $q_K^{-1}(W)\subseteq O$. Therefore, the mapping $q_K$ is closed.
\end{proof}

\begin{proposition}\label{pro3}
Suppose that $H$ and $K$ are subgroups in a topological group $G$ such that $H$ is closed and $K$ is neutral. If the quotient space $K\backslash G$ is first-countable, then the double coset space $K\backslash G/H$ is pseudometrizable. If, in addition, $K$ is strongly neutral, then the double coset space $K\backslash G/H$ is metrizable
\end{proposition}

\begin{proof}
Put $Z=K\backslash G/H$. According to Theorem \ref{th1}, $\mathscr{B}_Z^r$ is a base for some right uniformity and right separated uniformity $\mathcal {U}_Z^r$ on $Z$ whenever $K$ is neutral and strongly neutral, respectively. Thus, to show that the double coset space $K\backslash G/H$ is (pesudo)metrizable it is only to show that $\mathcal {U}_Z^r$ has a countable base by Theorem \ref{T1}.

Let $q:G\rightarrow K\backslash G$ and $\pi:G\rightarrow K\backslash G/H$ are the natural quotient mappings. Then there is a continuous open mapping $q_K: K\backslash G\rightarrow K\backslash G/H$ defined by $q_K(q(x))=\pi(x)$ for each $x\in G$, because $\pi$ and $q$ are continuous and open by Proposition \ref{pro1}.

\[
\xymatrix{ G\ar@{>}[d]_{q} \ar@{>}[dr]^{\pi}
  \\
 K\backslash G\ar@{>}[r]^{q_K}  & K\backslash G/H
}
\]

If $K\backslash G$ is first-countable, then let $\beta=\{V_n:n\in \omega\}$ be a family of open neighborhoods of $e$ such that $\{q(V_n):n\in \omega\}$ is a countable neighborhood base at $q(e)$, where $e$ is the identity in a topological group $G$. Then every element of the family $\gamma=\{O_n:n\in \omega\}$ is an open symmetric neighborhood of
$e$ in a topological group $G$, where $O_n=KV_n\cap (KV_n)^{-1}$ for each $V_n\in\beta$. We claim that the family $\{E_{O_n}^r: O_n\in \gamma\}\subseteq \mathscr{B}_Z^r$ is a base of $\mathcal {U}_Z^r$.

Since $\mathscr{B}_Z^r$ is a base of $\mathcal {U}_Z^r$, it is enough to show that arbitrary element of $\mathscr{B}_Z^r$ contains an element of $\{E_{O_n}^r: O_n\in \gamma\}$. Take arbitrary $E_V^r\in \mathscr{B}_Z^r$. Then $q(V)$ is an open neighborhood of $q(e)$, so there is a $V_n\in \beta$ such that $q(V_n)\subseteq q(V)$. This implies that $KV_n=q^{-1}(q(V_n))\subseteq q^{-1}(q(V))=KV$. Thus, $O_n=KV_n\cap (KV_n)^{-1}\subseteq KV_n\subseteq KV$. We claim that $E_{O_n}^r\subseteq E_V^r$. Indeed, take arbitrary point $(\pi(x),\pi(y))\in E_{O_n}^r$ for some $x,y\in G$. Then $\pi(y)\in \pi(O_nx)$, which implies that $KyH\subseteq KO_nxH$. By $O_n\subseteq KV$, we have that $$KyH\subseteq KO_nxH\subseteq KKVxH=KVxH,$$
which implies that $$\pi(y)\in\pi(KyH)\subseteq \pi(KVxH)= \pi(Vx).$$ Thus $(\pi(x),\pi(y))\in E_V^r$.
\end{proof}

In Proposition \ref{pro3}, when $K$ is the identity of $G$, we have the following result.
\begin{corollary}
Let $H$ be a closed subgroup of a first-countable topological group $G$ (without any separator axiom on $G$). Then the quotient space $G/H$ is metrizable.
\end{corollary}

\section{The coset spaces of generalized countably compact topological groups}
In this section, we shall study the coset spaces of generalized countably compact topological groups. First, recall some concepts as following:

\begin{lemma}\label{Lema1}
Let $H$ be a subgroup of a topological group $G$.
If $H$ is of countable character $\gamma = \{V_n : n\in \omega\}$ in $G$ such that $V^2_{n+1}\subseteq V_n$ for each $n\in \omega$, then $H$ is strongly neutral. Furthermore, the quotient mapping $\pi$ of $G$ onto the quotient space $H\backslash G$ is closed and open, and $H\backslash G$ is metrizable.
\end{lemma}

\begin{lemma}\label{Lema2} \cite[Lemma 1.7.6]{LY}
Let $\{U_n:n\in \omega\}$ be a decrease sequence of open sets in $X$ such that $\bigcap_{n\in \omega}U_n=\bigcap_{n\in \omega}\overline{U_n}$. Put $C=\bigcap_{n\in \omega}U_n$. Then the following statements are equivalent:
\begin{enumerate}
\item[(1)] each sequence $\{x_n:n\in\omega\}$ with $x_i\in U_i$ for each $i\in\omega$ has a cluster point in $X$;
\item[(2)] the family $\{U_n:n\in \omega\}$ is a base at $C$ in $X$ and $C$ is a countably compact set in $X$.
\end{enumerate}
\end{lemma}

Recall that a continuous mapping $f:X\rightarrow Y$ is {\it quasi-perfect} if $f$ is closed such that $f^{-1}(y)$ is countably compact for each $y\in Y$. Recall that a space is an {\it $M$-space} if it is a quasi-perfect preimage of a metrizable space \cite{M}. A point $x\in X$ is called a {\it $q$-point} of a space $X$ if there exists a sequence $\{U_n:n\in\omega\}$ of open neighborhoods
of $x$ in $X$ such that any sequence $\{x_n:n\in\omega\}$ of points in $X$ such that $x_n\in U_n$ for each $n\in\omega$ has a cluster point.
A space $X$ is said to be a {\it$q$-space} if every point of it is a $q$-point.

\begin{theorem}\label{th}
Suppose that $G$ is a topological group  with a $q$-point and $H$ is a closed subgroup of $G$. Then the quotient space $G/H$ is an open and quasi-perfect preimage of a metrizable space. In particular, $G/H$ is an $M$-space.
\end{theorem}

\begin{proof}
By the homogeneity of $G$, we can assume that the identity $e$ is a $q$-point. Then there is a sequence $\{U_n:n\in\omega\}$ of open neighborhoods of $e$ in $G$ satisfies that any sequence $\{x_n:n\in\omega\}$ of points such that $x_n\in U_n$ for each $n\in\omega$ has a cluster point in $G$. We define by induction a sequence $\{V_n:n\in\omega\}$ of symmetric open neighborhoods of $e$ in $G$ satisfying the following conditions:
\begin{enumerate}
\item[(1)] $V_0\subseteq  U_0$;
\item[(2)] $V_{n+1}^2\subseteq V_{n}$ and $V_{n+1}\subseteq U_{n+1}$ for each $n\in \omega$.
\end {enumerate}
Since $V_{n+1}^2\subseteq V_{n}$ and $V_n$ is a symmetric open neighborhood of $e$ for each $n\in \omega$, we have $\overline{V_{n+1}}\subseteq V_n$. One can easily show that the sequence $\{V_n:n\in \omega\}$ satisfies the condition (1) in Lemma \ref{Lema2}. Put $K=\bigcap_{n\in \omega}V_n$. Then $K$ is a closed and countably compact subgroup in $G$ such that the sequence $\{V_n:n\in \omega\}$ is a countable neighborhood base at $K$ in $G$ by Lemma \ref{Lema2}. Thus, by Lemma \ref{Lema1}, $K$ is a strongly neutral subgroup in $G$ and the quotient space $K\setminus G$ is metrizable.

Let $\pi:G\rightarrow G/H$ and $q:G\rightarrow K\backslash G/H$ be natural quotient mappings. Define a natural mapping $f:G/H\rightarrow K\backslash G/H$ by $f(\pi(x))=q(x)$ for each $x\in G$.

 \[
\xymatrix{ G\ar@{>}[r]^{\pi} \ar@{>}[dr]^{q}
& G/H \ar@{>}[d]^{f}   \\
   & K\backslash G/H
}
\]

 Since $K$ is strongly neutral and $K\setminus G$ is metrizable, $K\backslash G/H$ is metrizable by Proposition \ref{pro3}. By Proposition \ref{Pro1}, $f$ is a continuous, open and closed mapping from $G/H$ onto the metrizable space $K\backslash G/H$. To show that $f$ is quasi-perfect it is enough to prove that $f^{-1}(q(x))$ is countably compact for each $x\in G$. Indeed, since $f^{-1}(q(x))=\pi(q^{-1}q(x))=\pi(KxH)= \pi(Kx)$, $f^{-1}(q(x))$ is countably compact as a continuous image of the countably compact set $Kx$, because $K$ is countably compact.
\end{proof}

It is well known that every $M$-space is a $q$-space. Thus from Theorem \ref{th} it follows the following result:

\begin{corollary}\cite[Theorem 6]{PY}
A topological group $G$ is a $q$-space if and only if $G$ is an open quasi-perfect preimage of a metrizable.
\end{corollary}

A continuous mapping $f:X\rightarrow Y$ is {\it sequentially perfect} if $f$ is closed such that $f^{-1}(y)$ is sequentially compact for each $y$ in $Y$. A point $x\in X$ is called a {\it strict $q$-point} of a space $X$ if there exists a $q$-sequence $\{U_n:n\in\omega\}$ at $x$ such that $\bigcap_{n\in\omega} U_n$ is a sequentially compact set. We also call $\{U_n:n\in\omega\}$ a {\it strict $q$-sequence}. A space $X$ is said to be a {\it strict $q$-space} if every point of it is a strict $q$-point. It is obvious that any countably compact space which has no non-trivial convergent sequence is a $q$-space but not a strict $q$-space(for example, the Stone-\v{C}ech compactification of $\mathbb{N}$). A continuous mapping $f:X\rightarrow Y$ is called {\it sequentially perfect} if $f$ is closed and $f^{-1}(y)$ is a sequentially compact set for each $y\in Y.$

\begin{theorem}\label{th2}
Suppose that $G$ is a topological group  with a strict $q$-point and $H$ is a closed subgroup of $G$. Then the quotient space $G/H$ is an open and sequentially perfect preimage of a metrizable space.
\end{theorem}

\begin{proof}
By the homogeneity of $G$, we can assume that the identity $e$ is a strict $q$-point. Then there is a sequence $\{U_n:n\in\omega\}$ of open neighborhoods of $e$ in $G$ satisfies that any sequence $\{x_n:n\in\omega\}$ of points such that $x_n\in U_n$ for each $n\in\omega$ has a cluster point in $G$ and $\bigcap_{n\in \omega}U_n$ is sequentially compact. We define by induction a sequence $\{V_n:n\in\omega\}$ of symmetric open neighborhoods of $e$ in $G$ satisfying the following conditions:
\begin{enumerate}
\item[(1)] $V_0\subseteq  U_0$;
\item[(2)] $V_{n+1}^2\subseteq V_{n}$ and $V_{n+1}\subseteq U_{n+1}$ for each $n\in \omega$.
\end {enumerate}
Since $V_{n+1}^2\subseteq V_{n}$ and $V_n$ is a symmetric open neighborhood of $e$ for each $n\in \omega$, we have $\overline{V_{n+1}}\subseteq V_n$. One can easily show that the sequence $\{V_n:n\in \omega\}$ satisfies the condition (1) in Lemma \ref{Lema2}. Put $K=\bigcap_{n\in \omega}V_n$. Then $K$ is a closed and sequentially compact subgroup in $G$ such that the sequence $\{V_n:n\in \omega\}$ is a countable neighborhood base at $K$ in $G$ by Lemma \ref{Lema2}. Thus, by Lemma \ref{Lema1}, $K$ is a strongly neutral subgroup in $G$ and the quotient space $K\setminus G$ is metrizable and the natural quotient mapping $p:G\rightarrow K\setminus G$ is an open and closed mapping.

Let $\pi:G\rightarrow G/H$ and $q:G\rightarrow K\backslash G/H$ be natural quotient mappings. Define a natural mapping $f:G/H\rightarrow K\backslash G/H$ by $f(\pi(x))=q(x)$ for each $x\in G$.

\[
\xymatrix{ G\ar@{>}[r]^{\pi} \ar@{>}[dr]^{q}
& G/H \ar@{>}[d]^{f}   \\
   & K\backslash G/H
}
\]

Since $K$ is strongly neutral and $K\setminus G$ is metrizable, $K\backslash G/H$ is metrizable by Proposition \ref{pro3}. By Proposition \ref{Pro1}, $f$ is a continuous, open and closed mapping from $G/H$ onto the metrizable space $K\backslash G/H$. To show that $f$ is sequentially perfect it is enough to prove that $f^{-1}(q(x))$ is sequentially compact for each $x\in G$. Indeed, since $f^{-1}(q(x))=\pi(q^{-1}q(x))=\pi(KxH)= \pi(Kx)$, $f^{-1}(q(x))$ is sequentially compact as a continuous image of the sequentially compact set $Kx$. Note that $K$ is sequentially compact.
\end{proof}

%Since every sequentially compact set ia countably compact, every sequentially perfect mapping is quasi-perfect. It is well that every quasi-perfect preimage of a metrizable space is a $q$-space. Thus we have the following result:
\begin{proposition}\label{pp}
Let $f:X\rightarrow Y$ be a sequentially perfect mapping. If $Y$ is first-countable, then $X$ is a strict $q$-space.
\end{proposition}

\begin{proof}
Take a point $x\in X$ and assume that $\{U_n:n\in \omega\}$ is a base at $f(x)$ in $Y$. Then we claim that $\{f^{-1}(U_n):n\in \omega\}$ is a strict $q$-sequence at $x$, which implies that $X$ is a strict $q$-space.
Since $f$ is closed and continuous, one can easily show that the $\{f^{-1}(U_n):n\in \omega\}$ is a base at $f^{-1}(f(x))$. Also, $f^{-1}(f(x))=\bigcap_{n\in \omega}f^{-1}(U_n)$ is a sequentially compact set, because $f$ is sequentially perfect. Take any point $x_n\in f^{-1}(U_n)$ for each $n\in\omega$. If $\bigcap_{n\in \omega}f^{-1}(U_n)$ contain a subsequence of the sequence $\{x_n\}_{n\in\omega},$ then one can easily show that the sequence $\{x_n\}_{n\in\omega}$ has a convergent subsequence, because $f^{-1}(f(x))=\bigcap_{n\in \omega}f^{-1}(U_n)$ is a sequence compact set. This implies that the sequence $\{x_n\}_{n\in\omega}$ has an accumulation point. If $\bigcap_{n\in \omega}f^{-1}(U_n)$ contain no subsequence of the sequence $\{x_n\}_{n\in\omega},$ without loss of generality, assuming that $x_n\in f^{-1}(U_n)\setminus \bigcap_{n\in \omega}f^{-1}(U_n)$ for each $n\in\omega$, then one can easily show that that the sequence $\{x_n\}_{n\in\omega}$ has an accumulation point. Indeed, if not, then $\{x_n : n\in\omega\}$ is a closed set. $\{f(x_n ): n\in\omega\}$ is also a closed set in $Y$, because $f$ is closed. Then $Y\setminus\{f(x_n ): n\in\omega\}$ is an open set containing $f(x)$. However, there is no element of $\{U_n:n\in\omega\}$ contained in $Y\setminus\{f(x_n ): n\in\omega\}$. This is contradict with the fact that $\{U_n:n\in \omega\}$ is a base at $f(x)$ in $Y$. Thus we have proved that $\{f^{-1}(U_n):n\in \omega\}$ is a strict $q$-sequence.
\end{proof}

From Theorem \ref{th2} and Proposition \ref{pp} it follows that:
\begin{corollary}\cite[Theorem 2.6]{LX}
A topological group $G$ is a strict $q$-space if and only if $G$ is an open and sequentially perfect preimage of a metrizable space.
\end{corollary}

A continuous mapping $f:X\rightarrow Y$ is {\it strongly sequentially perfect} if $f$ is closed such that $f^{-1}(F)$ is sequentially compact for each sequentially compact $F$ in $Y$. A point $x\in X$ is called a {\it strong $q$-point} of a space $X$ if there exists a sequence $\{U_n:n\in\omega\}$ of open neighborhoods
of $x$ in $X$ such that any sequence $\{x_n:n\in\omega\}$ of points in $X$ such that $x_n\in U_n$ for each $n\in\omega$ has a convergent subsequence.
A space $X$ is said to be a {\it strong $q$-space} if every point of it is a strong $q$-point \cite{PY}.

Let $A$ be a subset of a space $X$ and $\gamma=\{U_n:n\in \omega\}$ be a sequence of open neighborhoods containing $A$. Recall that $\gamma$ is a {\it countable neighborhood base at $A$} if for each open set $U$ containing $A$ there is a $U_n\in \gamma$ such that $U_n\subseteq U$. $\gamma$ is called a {\it strongly countable neighborhood base at $A$}, if $\gamma$ is a countable neighborhood base at $A$ such that each sequence $\{x_n:n\in \omega\}$ with $x_n\in U_n$ for each $U_n\in \gamma$ has a convergent subsequence.

\begin{lemma}\label{LL}
Let $f:X\rightarrow Y$ be a continuous and open mapping. Then $f(A)$ has a strongly countable neighborhood base for each subset $A\subseteq X$ with a strongly countable neighborhood base.
\end{lemma}

\begin{proof}
Let $\{W_n:n\in\omega\}$ be a strongly countable neighborhood base of $A$. Then one can easily show that $\{f(W_n):n\in\omega\}$ is a strongly countable open neighborhood base of $f(A)$, because $f$ is a continuous and open mapping. Take a sequence $\{y_n:n\in \omega\}$ with $y_n\in f(W_n)$ for each $n\in\omega$. For each $y_n$, take an $x_n\in W_n$ such that $f(x_n)=y_n$. Then the sequence $\{x_n:n\in \omega\}$ has a convergent subsequence $\{x_{n_i}:i\in \omega\}$, because $\{W_n:n\in\omega\}$ is a strongly countable open neighborhood base of $A$. Then, by the continuity of $f$, one can obtain that $\{f(x_{n_i}):i\in \omega\}$ is convergent subsequence of  $\{y_n:n\in \omega\}$. This finishes the proof.
\end{proof}

\begin{theorem}\label{th3}
Suppose that $G$ is a topological group with a strong $q$-point and $H$ is a closed subgroup of $G$. Then the quotient space $G/H$ is an open and strongly sequentially perfect preimage of a metrizable space.
\end{theorem}

\begin{proof}
By the homogeneity of $G$, we can assume that the identity $e$ is a strong $q$-point. Then there is a sequence $\{U_n:n\in\omega\}$ of open neighborhoods of $e$ in $G$ satisfies that any sequence $\{x_n:n\in\omega\}$ of points such that $x_n\in U_n$ for each $n\in\omega$ has a convergent subsequence. We define by induction a sequence $\{V_n:n\in\omega\}$ of symmetric open neighborhoods of $e$ in $G$ satisfying the following conditions:
\begin{enumerate}
\item[(1)] $V_0\subseteq  U_0$;
\item[(2)] $V_{n+1}^2\subseteq V_{n}$ and $V_{n+1}\subseteq U_{n+1}$ for each $n\in \omega$.
\end {enumerate}
Since $V_{n+1}^2\subseteq V_{n}$ and $V_n$ is a symmetric open neighborhood of $e$ for each $n\in \omega$, we have $\overline{V_{n+1}}\subseteq V_n$. One can easily show that the sequence $\{V_n:n\in \omega\}$ satisfies the condition (1) in Lemma \ref{Lema2}. Put $K=\bigcap_{n\in \omega}V_n$. Then $K$ is a closed sequentially compact subgroup in $G$ such that the sequence $\gamma=\{V_n:n\in \omega\}$ is a countable neighborhood base at $K$ in $G$ by Lemma \ref{Lema2}. Particularly, $\gamma$ is a strongly countable neighborhood base at $K$ in $G$. Thus, by Lemma \ref{Lema1}, $K$ is a strongly neutral subgroup in $G$ and the quotient space $K\setminus G$ is metrizable.

Let $\pi:G\rightarrow G/H$ and $q:G\rightarrow K\backslash G/H$ be natural quotient mappings. Define a natural mapping $f:G/H\rightarrow K\backslash G/H$ by $f(\pi(x))=q(x)$ for each $x\in G$.

\[
\xymatrix{ G\ar@{>}[r]^{\pi} \ar@{>}[dr]^{q}
& G/H \ar@{>}[d]^{f}   \\
   & K\backslash G/H
}
\]

Since $K$ is strongly neutral and $K\setminus G$ is metrizable, $K\backslash G/H$ is metrizable by Proposition \ref{pro3}. By Proposition \ref{Pro1}, $f$ is a continuous, open and closed mapping from $G/H$ onto the metrizable space $K\backslash G/H$.

To show that $f$ is strongly sequentially perfect it is enough to prove that $f^{-1}(F)$ is sequentially compact for each sequentially compact set $F\subseteq K\backslash G/H$. Indeed, take arbitrary sequence $\{x_n:n\in \omega\}\subseteq f^{-1}(F)$. Then there is a subsequence of $\{f(x_n):n\in \omega\}$ converging to some point $z_0\in F$, because the sequence $\{f(x_n):n\in \omega\}$ is contained in the sequentially compact set $F$. Without loss of generality, we assume that $\{f(x_n):n\in \omega\}$ converges to the point $z_0$. Take a $y\in G$ such that $q(y)=z_0$. Clearly, $Ky$ has a strongly countable neighborhood base, because $K$ has a strongly countable open neighborhoods base. Since $\pi$ is the natural quotient mapping, $\pi$ is continuous and open. By Lemma \ref{LL},  $f^{-1}(z_0)=\pi(q^{-1}(q(y)))=\pi(KyH)=\pi(Ky)$ as a continuous open image of the set $Ky$ has a strongly countable neighborhood base $\{W_n:n\in \omega\}$. Since $f$ is closed and $f^{-1}(z_0)\subseteq W_n$ for each $n\in \omega$, we can assume that $f^{-1}(f(W_n))=W_n$ for each $n\in \omega$. Then we claim that $W_i$ contains infinite elements of the sequence $\{x_n:n\in \omega\}$ for each $i\in \omega$. Indeed, if not, there is an $i_0\in \omega$ such that $f(W_{i_0})$ contains finite elements of the sequence $\{f(x_n):n\in \omega\}$. This contradicts with the sequence $\{f(x_n):n\in \omega\}$ converging to the point $z_0$, because $f(W_{i_0})$ as an open image of the open set $W_{i_0}$ is an open neighborhood of $z_0$. Thus one can take a subsequence $\{x_{n_i}:i\in \omega\}$ of the sequence $\{x_n:n\in \omega\}$ such that $x_{n_i}\in W_i$ for each $i\in \omega$. Since $\{W_n:n\in \omega\}$ is a strongly countable neighborhood base at $f^{-1}(z_0)$, the sequence $\{x_{n_i}:i\in \omega\}$ has a convergent subsequence. This implies that the sequence $\{x_n:n\in \omega\}$ contains a convergent subsequence, because $\{x_{n_i}:i\in \omega\}$ is a subsequence of $\{x_n:n\in \omega\}$. Let $\{x_{n_k}:k\in \omega\}$ be a subsequence of $\{x_n:n\in \omega\}$ such that it converges to a point $x_0\in G/H$ and $x_{n_k}\in W_k$ for each $k\in \omega$. Since $H$ is closed, the quotient space $G/H$ is a Hausdorff space. Thus one can easily show that $f^{-1}(z_0)=\bigcap_{n\in\omega}W_n$, because the sequence $\{W_n:n\in \omega\}$ is a neighborhood base at $f^{-1}(z_0)$. Thus we claim that $x_0\in f^{-1}(z_0)\subseteq f^{-1}(F)$. If not, the set $f^{-1}(z_0)$ contains finite elements of the sequence $\{x_{n_k}:k\in \omega\}$. Without loss generality, we assume that $\{x_{n_k}:k\in \omega\}\cap f^{-1}(z_0) =\emptyset$. Then $(G/H)\setminus (\{x_{n_k}:k\in \omega\}\cup\{x_0\})$ is an open neighborhood of $f^{-1}(z_0)$ such that contains no $W_n$ for each $n\in \omega$. This is a contradiction with $\{W_n:n\in \omega\}$ being a neighborhood base at $f^{-1}(z_0)$. We have proved that $f^{-1}(F)$ is sequentially compact for each sequentially compact set $F\subseteq K\backslash G/H$.
\end{proof}

\begin{proposition}\label{LLL}
Let $f:X\rightarrow Y$ be an open and strongly sequentially perfect mapping. Then $X$ is a strong $q$-space if and only if so is $Y$.
\end{proposition}

\begin{proof}
Assume that $X$ is  a strong $q$-space. Take a point $x\in X$ and a strong $q$-sequence $\{U_n:n\in \omega\}$ at $x$. Then one can easily show that $\{f(U_n):n\in \omega\}$ is a strong $q$-sequence at $f(x)$, because $f$ is continuous and open. This shows that $Y$ is a strong $q$-space.

Conversely, take a point $x\in X$ and a strong $q$-sequence $\{U_n:n\in \omega\}$ at $f(x)$. We claim that the family $\{f^{-1}(U_n):n\in \omega\}$ is a strong $q$-sequence at $x$, which implies that $X$ is a strong $q$-space.
Indeed, take $x_n\in f^{-1}(U_n)$ for each $n\in \omega$. Then $f(x_n)\in U_n$ for each $n\in \omega$. Since $\{U_n:n\in \omega\}$ is a strong $q$-sequence at $f(x)$, $\{f(x_{n})\}_{n\in\omega}$ has a convergent subsequence $\{f(x_{n_i})\}_{i\in\omega}$. Put $F=\{f(x_{n_i}):i\in \omega\}\cup \{y_0\}$, where $y_0$ is a convergent point of the sequence $\{f(x_{n_i})\}_{i\in\omega}$. Then $f^{-1}(F)$ is a sequentially compact set in $X$, because $f$ is a strongly sequentially perfect mapping. Clearly, the subsequence $\{x_{n_i}\}_{i\in\omega}$ of the sequence $\{x_{n}\}_{n\in\omega}$ is contained in $f^{-1}(F)$, so $\{x_{n_i}\}_{i\in\omega}$ has a convergent subsequence, which is also a convergent subsequence of the sequence $\{x_{n}\}_{n\in\omega}$. This shows that the family $\{f^{-1}(U_n):n\in \omega\}$ is a strong $q$-sequence at $x$.
\end{proof}

From Theorem \ref{th3} and Proposition \ref{LLL} it follows that:

\begin{corollary}\cite[Theorem 2.5]{LX}
A topological group $G$ is a strong $q$-space if and only if $G$ is an open and strongly sequentially perfect preimage of a metrizable.
\end{corollary}

\begin{proposition}\label{pro4.1}
Let $(X, d)$ be a pseudometric space. Then the natural mapping $p:X\rightarrow X/d$ is open and satisfies that $p^{-1}(F)$ is $r$-pseudocompact in $X$ for each $r$-pseudocompact set $F$ in $X/d$, where the topology on $X$ is induced by $d$ and $X/d$ is the equivalent classes by $d$, which is endowed with the quotient topology.
\end{proposition}

\begin{proof}

\end{proof}

A point $x\in X$ is called {\it pseudocompacness point} of $X$ if there exists a sequence $\{U_n:n\in \omega\}$ of open neighborhoods of $x$ in $X$ satisfying the condition: (pp) every sequence $\{V_n:n\in \omega\}$ of non-empty open sets in $X$ such that $V_n\subseteq U_n$ for each $n\in \omega$ has a point of accumulation in $X$. A space is said to be {\it pointwise pseudocompact} \cite{AT1} if each point of $X$ is a pseudocompacness point.

Recall that a subset $A$ of $X$ is said to be a {\it relative pseudocompact}, for brevity, {\it $r$-pseudocompact} \cite{HST} in $X$ if every infinite family $\gamma$ of open sets in $X$ such that $U\cap A\neq\emptyset$, for each $U\in \gamma$, has an accumulation point in $A$.

\begin{theorem}\label{th4}
Suppose that $G$ is a topological group  with a pseudocompact point and $H$ is a closed subgroup of $G$. Then there is a metrizable space $M$ and a continuous open mapping $f$ from $G/H$ onto $M$ such that $f^{-1}(F)$ is $r$-pseudocompact in $G/H$ for each $r$-pseudocompact set $F$ in $M$.
\end{theorem}

\begin{proof}
By the homogeneity of $G$, we can assume that the identity $e$ is a pseudocompact point. Then there is a sequence $\{U_n:n\in\omega\}$ of open neighborhoods of $e$ in $G$ satisfying the condition: (pp) every sequence $\{W_n:n\in \omega\}$ of non-empty open sets in $G$ such that $W_n\subseteq U_n$ for each $n\in \omega$ has a point of accumulation in $G$. We define by induction a sequence $\{V_n:n\in\omega\}$ of symmetric open neighborhoods of $e$ in $G$ satisfying the following conditions:
\begin{enumerate}
\item[(1)] $V_0\subseteq  U_0$;
\item[(2)] $V_{n+1}^2\subseteq V_{n}$ and $V_{n+1}\subseteq U_{n+1}$ for each $n\in \omega$.
\end {enumerate}

Put $K=\bigcap_{n\in \omega}V_n$. Since $V_{n+1}^2\subseteq V_{n}$ and $V_n$ is a symmetric open neighborhood of $e$ for each $n\in \omega$, one can easily show that $K$ is a closed $r$-pseudocompact subgroup in $G$.

Let $q:G\rightarrow K\backslash G$ and $\pi:G\rightarrow K\backslash G/H$ are the natural quotient mappings. Then there is a continuous open mapping $q_K: K\backslash G\rightarrow K\backslash G/H$ defined by $q_K(q(x))=\pi(x)$ for each $x\in G$, because $\pi$ and $q$ are continuous and open by Proposition \ref{pro1}.

\[
\xymatrix{ G\ar@{>}[d]_{q} \ar@{>}[dr]^{\pi}
  \\
 K\backslash G\ar@{>}[r]^{q_K}  & K\backslash G/H
}
\]

In the proof of (1)$\Rightarrow$ (2) in \cite[Theorem 2.9]{LX}, it has proved that $K\backslash G$ is metrizable. Since $K$ is an $r$-pseudocompact subgroup, $K$ is a bounded subgroup in $G$ by \cite[Proposition 2.1]{HST}. It follows that $K$ is precompact from the fact that  every bounded subset of a topological group is precompact \cite[Proposition 6.10.2]{AT1}. By \cite [Lemma 2.8]{LX} , $K$ is a neutral subgroup. Thus, from Proposition \ref{pro3} it follows that $K\backslash G/H$ is pseudometrizable.

Let $\pi_1:G\rightarrow G/H$ and $p:G\rightarrow K\backslash G/H$ be natural quotient mappings. Define a natural mapping $f:G/H\rightarrow K\backslash G/H$ by $f(\pi_1(x))=p(x)$ for each $x\in G$.

\[
\xymatrix{ G\ar@{>}[r]^{\pi_1} \ar@{>}[dr]^{p}
& G/H \ar@{>}[d]^{f}   \\
   & K\backslash G/H
}
\]

By Proposition \ref{pro4.1}, it is enough to show that $f$ is open and satisfies that $f^{-1}(F)$ is $r$-pseudocompact in $G/H$ for each $r$-pseudocompact set $F$ in $ K\backslash G/H$.

Clearly, $f$ is open. Suppose that the infinite family $\gamma=\{O_n:n\in\omega\}$ of open sets in $G/H$ such that $O_n\cap f^{-1}(F)\neq\emptyset$, for each $O_n\in \gamma$. Then the infinite family $\gamma_1=\{f(O_n):n\in\omega\}$ of open sets in $K\backslash G/H$ has an accumulation point $z_0$ in $F$, because $F$ is $r$-pseudocompact in $ K\backslash G/H$. Choose a point $x_0$ in $G$ such that $p(x_0)=z_0$. We claim that $\pi_1(KV_nx_0)$ joints with infinite elements in $\gamma$ for each $n\in\omega$. Indeed, if not, there is $\pi_1(KV_{n_0}x_0)$ joints with only finite elements in $\gamma$. Then $f(\pi_1(KV_{n_0}x_0))$ is an open neighborhood of $z_0$ such that $f(\pi_1(KV_{n_0}x_0))$ joints with only finite elements in $\gamma_1$, because

\begin{align*}
f^{-1}(f(\pi_1(KV_{n}x_0)))&=f^{-1}(p(KV_{n}x_0))\\
&=\pi_1(p^{-1}(p(KV_{n}x_0)))\\
&=\pi_1(KV_{n}x_0H)\\
&=\pi_1(KV_{n}x_0)
\end{align*}
holds for each $n\in \omega$. This is a contradiction with $z_0$ being an accumulation point of $\gamma_1$. Without loss of generality, we assume that $\pi_1(KV_{n}x_0)\cap O_n\neq\emptyset$ holds for each $n\in \omega$. One can easily show that the family $\{KV_n:n\in \omega\}$ satisfies the condition (pp), so does the family $\{KV_nx_0:n\in \omega\}$. Thus the family $\{KV_{n}x_0\cap \pi_1^{-1}(O_n): n\in \omega\}$ has an accumulation point $y_0$ in $G$, because $KV_{n}x_0\cap \pi_1^{-1}(O_n)$ is open and contained in $KV_nx_0$ for each $n\in \omega$. We claim that $y_0\in Kx_0$. Indeed, since $$\overline{KV_{n+2}x_0}\subseteq \overline{V_{n+2}^2x_0}\subseteq \overline{V_{n+1}x_0} \subseteq  V_{n+1}V_{n+1}x_0 \subseteq  V_nx_0$$
holds for each $n\in\omega$, we have that $$y_0\in\bigcap_{n\in \omega}\overline{KV_{n}x_0} \subseteq \bigcap_{n\in \omega}V_nx_0= (\bigcap_{n\in \omega}V_n)x_0=Kx_0.$$
Thus, $\pi_1(y_0)$ is an accumulation point of the family $\{\pi_1(KV_{n}x_0)\cap O_n: n\in \omega\}$ contained in $\pi_1(Kx_0)$. Since, by $z_0\in F$, $$f^{-1}(z_0)=\pi_1(p^{-1}(z_0))=\pi_1(Kx_0H)=\pi_1(Kx_0)\subseteq f^{-1}(F),$$
we have that $\pi_1(y_0)\in f^{-1}(F)$. Thus we have proved that the family $\{O_n:n\in \omega\}$ has an accumulation point $\pi_1(y_0)$ contained in $f^{-1}(F)$.
\end{proof}

\begin{proposition}\label{pro3.2}
Let $f$ be a continuous open mapping from $X$ onto a first-countable space $Y$ such that $f^{-1}(F)$ is $r$-pseudocompact in $X$ for each $r$-pseudocompact set $F$ in $Y$. Then $X$ is a pointwise pseudocompact-space.
\end{proposition}

\begin{proof}
Take a point $x\in X$ and assume that a decreasing family $\{U_n:n\in\omega\}$ is a base at $f(x)$. We claim that the family $\{f^{-1}(U_n):n\in \omega\}$ satisfies the condition (pp): every sequence $\{V_n:n\in \omega\}$ of non-empty open sets in $X$ such that $V_n\subseteq f^{-1}(U_n)$ for each $n\in \omega$ has a point of accumulation in $X$. Clearly, $f(x)$ is an accumulation of the family $\{f(V_n):n\in \omega\}$. Then one can find a sequence $\{y_n\}_{n\in \omega}$ in $Y$ such that the sequence $\{y_n\}_{n\in \omega}$ converges to the point $f(x)$ and joints with infinitely elements of the family $\{f(V_n):n\in \omega\}$. Put $F=\{y_n:n\in \omega\}\cup \{f(x)\}$. Clearly, the set $F$ is $r$-pseudocompact in $Y$, so $f^{-1}(F)$ is an $r$-pseudocompact set in $X$ such that $f^{-1}(F)$ joints with infinitely elements of the family $\{V_n:n\in \omega\}$. Thus, there is an accumulation of the family $\{V_n:n\in \omega\}$ in $f^{-1}(F)$. This completes the proof.
\end{proof}

From Theorem \ref{th4} and Proposition \ref{pro3.2} it follows the following result:

\begin{corollary}\cite[Theorem 2.9]{LX}
A topological group $G$ is a pointwise pseudocompact-space if and only if there is a continuous open mapping $f$ from $G$ onto a metrizable space $M$ such that $f^{-1}(F)$ is a $r$-pseudocompact set in $G$ for each $r$-pseudocompact set $F$ in $M$.
\end{corollary}

\section{Other related results}

\begin{theorem}\label{TT}
Let $H$ be a closed neutral subgroup of a topological group $G$ such that $H$ contains a subgroup $K$ such that the quotient space $G/K$ is first-countable. Then the quotient space $G/H$ is metrizable.
\end{theorem}

\begin{proof}
Since $H$ is a closed neutral subgroup, by (2) in Remark \ref{R} the topology on the quotient space $G/H$ is induced by the separated uniformity $\mathcal {U}_{G/H}^l$ and $\mathscr{B}_{G/H}^l$ is a base for $\mathcal {U}_{G/H}^l$. Thus, to show that the quotient space $G/H$ is metrizable, it is only to show that $\mathcal {U}_{G/H}^l$ has a countable base by Theorem \ref{T1}.
Let $q:G\rightarrow G/K$ and $\pi:G\rightarrow G/H$ are the natural quotient mapping. Since the quotient space $G/K$ is first-countable, there is a countable family $\gamma=\{O_n:n\in \omega\}$ of symmetric open neighborhoods of the identity $e$ in $G$ such that $\{q(O_n):n\in \omega\}$ is a base at $q(e)$. We claim that every element of $\mathscr{B}_{G/H}^l$ contains an element of $\{E_{O_n}^l:O_n\in \gamma\}$, where $E_{O_n}^l=\{(\pi(x),\pi(y):)\pi(y)\in \pi(xO_n)\}$ for all $n$. This implies that $\mathcal {U}_{G/H}^l$ has a countable base, because $\mathscr{B}_{G/H}^l$ is a base for $\mathcal {U}_{G/H}^l$.

Indeed, take arbitrary $E_V^l\in \mathscr{B}_{G/H}^l$, where $E_V^l=\{(\pi(x),\pi(y):)\pi(y)\in \pi(xV)\}$ and $V\in \mathscr{N}_s(e)$. Then there is $O_n\in \gamma$ such that $q(O_n)\subseteq q(V)$, which implies that $O_nK\subseteq VK$. We claim that $E_{O_n}^L\subseteq E_V^l$. Take any $(\pi(x),\pi(y))\in E_{O_n}^L$, which implies that $\pi(y)\in \pi(xO_n)$. Thus, we have that $y\in yH\subseteq xO_nH$. Since we have proved that $O_nK\subseteq VK$, we have that $xO_nK\subseteq xVK$, which implies that $xO_nKH\subseteq xVKH$. This implies that $xO_nH\subseteq xVH$ by $K\subseteq H$. Therefore, we have that $y\in xO_nH\subseteq xVH$, which implies that $\pi(y)\in \pi(xVH)=\pi(xV)$. Thus $(\pi(x),\pi(y))\in E_V^l$.
\end{proof}

Since the quotient space $H\backslash G$ and $G/H$ are homeomorphic for every subgroup $H$ in a topological group $G$, from Theorem \ref{TT} it follows:
\begin{corollary}\cite[Theorem 3.1]{FST}
Let $H$ be a closed neutral subgroup of a topological
group $G$. If the space $H\backslash G$ is first-countable, then it is metrizable.
\end{corollary}

\begin{corollary}
Let $H$ be a closed neutral subgroup of a pointwise pseudocompact topological group $G$. Then the quotient space $G/H$ is metrizable if and only if $G/H$ has a $G_\delta$-point.
\end{corollary}

\begin{corollary}\label{CC}
Let $H$ be a closed neutral subgroup of a topological group $G$ with a q-point. Then the quotient space $G/H$ is metrizable if and only if $G/H$ has a $G_\delta$-point.
\end{corollary}

In 2021, X. Ling et al. proved that: Let $H$ be a closed subgroup of a feathered group $G$; then the quotient space $G/H$ is metrizable
if and only if $G/H$ is first-countable \cite[Theorem 4.4]{LHL}. However, the proof seems to have a gap, because there is a separated uniformity with no countable base but it induces a metrizable topology (see \cite[Example 8.1.7]{En}). From Remarks \ref{R} and \ref{r} it follows that $\mathcal {U}_{G/K}^r$ and $\mathcal {U}_{G/K}^l$ are separated uniformities on the quotient space $G/K$ which induce the topology of $G/K$ in \cite[Theorem 4.4]{LHL}, but $\mathcal {U}_{G/K}^r$ may be have no countable base.  From Corollary \ref{CC} we have the following result:

\begin{corollary}
Let $H$ be a closed neutral subgroup of a feathered group $G$. Then the quotient space $G/H$ is metrizable
if and only if $G/H$ has a $G_\delta$-point..
\end{corollary}

%\smallskip

\section*{References}

%\bibliographystyle{elsarticle-num}
%\bibliography{database0110}

\def\cprime{$'$}

\end{document}